\documentclass[12pt]{article}
\usepackage{amssymb,amsthm,amsmath,latexsym}
\newtheorem{thm}{Theorem}
\newtheorem{prop}[thm]{Proposition}
\newtheorem{lem}[thm]{Lemma}

\theoremstyle{remark}
\newtheorem{rem}[thm]{Remark}
\newcommand{\FF}{\mathbb{F}}
\newcommand{\ZZ}{\mathbb{Z}}

\newcommand{\cB}{\mathcal{B}}
\newcommand{\cD}{\mathcal{D}}
\newcommand{\cP}{\mathcal{P}}
\DeclareMathOperator{\Aut}{Aut}

\begin{document}

\title{On the Classification of Weighing Matrices and 
Self-Orthogonal Codes}

\author{
Masaaki Harada\thanks{
Department of Mathematical Sciences,
Yamagata University,
Yamagata 990--8560, Japan, and
PRESTO, Japan Science and Technology Agency, Kawaguchi,
Saitama 332--0012, Japan. 
email: mharada@sci.kj.yamagata-u.ac.jp}
and 
Akihiro Munemasa\thanks{
Graduate School of Information Sciences,
Tohoku University,
Sendai 980--8579, Japan.
email: munemasa@math.is.tohoku.ac.jp}
}

\maketitle

\begin{abstract}
We provide a classification method of
weighing matrices based on a classification of self-orthogonal 
codes.
Using this method,
we classify weighing matrices of
orders up to $15$ and order $17$, by revising
some known classification.
In addition, we give a revised classification of weighing matrices
of weight $5$.
A revised classification of ternary maximal 
self-orthogonal codes of lengths $18$ and $19$
is also presented.
\end{abstract}

\section{Introduction}\label{Sec:1}

A weighing matrix $W$ of order $n$ and weight $k$ is an 
$n \times n$ $(1,-1,0)$-matrix $W$ such that 
$W W^T=kI_n$, where $I_n$ is 
the identity matrix of order $n$ and $W^T$ denotes the transpose of $W$.  
A weighing matrix of order $n$ and weight $n$
is also called a Hadamard matrix.  
We say that two weighing matrices $W_1$ and $W_2$ of order $n$ and 
weight $k$ are {\em equivalent} if there exist 
$(1,-1,0)$-monomial matrices $P$ and $Q$ with 
$W_1=P W_2 Q$.

Chan, Rodger and Seberry~\cite{CRS86} began a classification
of weighing matrices and they classified
all weighing matrices of weight $k \le 5$ and all
weighing matrices of orders $n \le 11$.
Ohmori~\cite{O89} and \cite{O93} classified
weighing matrices of orders $12$ and $13$, respectively.
At order $14$, weighing matrices of weights $k \le 8$ and
$13$ were classified in~\cite{CRS86} and~\cite{O92}, respectively.
At order $17$, all weighing matrices of weight $9$ with
intersection number $8$ were classified in~\cite{OM98}.

In this paper, we 
extend the classification of weighing matrices using the 
known classification of self-orthogonal codes.
Let $\ZZ_{m}$ be the ring of integers modulo $m$, where $m$
is an integer greater than $1$.
Let $W$ be a weighing matrix of order $n$ and weight $k$,
and suppose that $m$ is a divisor of $k$. 
If we regard the entries of $W$ as elements of $\ZZ_m$,
then the rows of $W$ generate a self-orthogonal $\ZZ_m$-code.
This means that $W$ can be regarded as a subset of codewords
in some maximal self-orthogonal code. For example, a classification
of weighing matrices of order $16$ and weight $6$ can be derived
from the known classification of ternary self-dual codes of length
$16$ given in~\cite{CPS}.

The paper is organized as follows. In Section~\ref{Sec:code},
we review the known classification of maximal self-orthogonal codes
needed for our classification of weighing matrices. It turns out
that there are errors in the classification of ternary maximal
self-orthogonal codes of lengths $18$ and $19$ given in~\cite{PSW},
and we correct them.
In Section~\ref{Sec:Method}, we give a detailed description of our
classification method of
weighing matrices of order $n$ and weight $k$
based on the classification of self-orthogonal $\ZZ_m$-codes
of length $n$, where $m$ is a divisor of $k$.
Our method, applied to the known classification of self-dual 
$\FF_5$-codes of length $12$, 
leads to a classification of weighing
matrices of order $12$ and weight $5$. This reveals an omission
in the classification given in~\cite[Theorem 5]{CRS86}, and 
a revised classification of 
weighing matrices of weight $5$ for all orders
is given in Section~\ref{Sec:wt5}, while 
a revised classification of 
weighing matrices of order $12$ for all weights
is given in Section~\ref{Sec:12}.
In Section~\ref{Sec:141517}, we classify weighing 
matrices of orders $14,15$ and $17$.
Again, there is an error in the number of weighing matrices of 
order $14$ and weight $8$ given in~\cite[Theorem~3]{O89},
and we correct it.
This completes a classification of weighing matrices of 
orders $n \le 17$ except $n = 16$.
Weighing matrices of order $n$ and
$k$ are also classified for
\[
(n,k)=(16,6), (16,9), (16,12), (18,9)
\]
in Section~\ref{Sec:other}.
All weighing matrices given in this paper
can be obtained electronically from~\cite{Data}.

\section{Maximal self-orthogonal codes}
\label{Sec:code}
\subsection{Codes}
We shall exclusively deal with the case 
$\ZZ_p=\FF_p$ and $\ZZ_4$, 
where $\FF_p$ denotes the finite field of odd prime order $p$.
A $\ZZ_{m}$-code $C$ of length $n$
(or a code $C$ of length $n$ over $\ZZ_{m}$)
is a $\ZZ_{m}$-submodule of $\ZZ_{m}^n$.
The dual code $C^\perp$ of $C$ is defined as
$C^\perp = \{ x \in \ZZ_{m}^n \mid x \cdot y = 0$ for all $y \in C\}$
under the standard inner product $x \cdot y$.
A code $C$ is {\em self-dual} if $C=C^\perp$, and 
$C$ is {\em self-orthogonal} if $C \subset C^\perp$.
A self-dual $\FF_p$-code of length $n$ exists if and only if
$n$ is even for $p \equiv 1 \pmod 4$, and
$n \equiv 0 \pmod 4$ for $p \equiv 3 \pmod 4$.
A self-dual $\ZZ_4$-code exists for every length. 

A self-orthogonal code $C$ is {\em maximal} if $C$ is
the only self-orthogonal code containing $C$.
The dimension of a maximal self-orthogonal $\FF_p$-code of
length $n$ is a constant
depending only on $n$ and $p$, and 
a self-dual code is automatically maximal.
More precisely, 
for $p \equiv 1 \pmod 4$, 
a maximal self-orthogonal $\FF_p$-code of length $n$
has dimension $(n-1)/2$ if $n$ is odd.
For $p \equiv 3 \pmod 4$,
a maximal self-orthogonal $\FF_p$-code of length $n$
has dimension $(n-1)/2$ if $n$ is odd,
$n/2-1$ if $n \equiv 2 \pmod 4$.
It is easy to see that
a maximal self-orthogonal $\ZZ_4$-code is necessarily self-dual
for every length.


Two codes $C$ and $C'$ are {\em equivalent} 
if there exists a $(1,-1,0)$-monomial matrix $P$ with 
$C' = C P = \{ x P\:|\: x \in C\}$.
The {\em automorphism group} $\Aut(C)$ of $C$ is the group of all
$(1,-1,0)$-monomial matrices $P$ with
$C = C P$.
Our classification method of weighing matrices of 
order $n$ and weight $k=mt$
requires a classification of maximal self-orthogonal
$\ZZ_m$-codes of length $n$  (see Section~\ref{Sec:Method}).
In this paper, some classifications of maximal self-orthogonal
$\ZZ_m$-codes are used for $m=3,4,5,7$ to classify weighing matrices.
The current knowledge on the classifications of such codes
is listed in Table~\ref{Tab:MSO}.

\begin{table}[thb]
\caption{Maximal self-orthogonal $\ZZ_m$-codes of length $n$}
\label{Tab:MSO}
\begin{center}
{\small
\begin{tabular}{l|l|l}
\noalign{\hrule height0.8pt}
$\ZZ_m$ & \multicolumn{1}{c|}{Lengths $n$} &
\multicolumn{1}{c}{References} \\
\hline
$\FF_3$ & $1,\ldots,12$  & \cite{MPS}   \\
        & $13,\ldots,16$ & \cite{CPS}  \\
        & $17,\ldots,20$ & \cite{PSW} (see also this section) \\
        & $24$           & \cite{HM09}  \\
\hline
$\ZZ_4$ & $1,\ldots,9$   & \cite{C-S-Z4} \\
        & $10,\ldots,15$ & \cite{FGLP-Z4} \\ 
        & $16$ (Type~II) & \cite{PLF-Z4} \\
        & $16$ (Type~I), $17,18,19$     & \cite{HM09-2} \\
\hline
$\FF_5$ & $1,\ldots,12$  & \cite{LPS}   \\
        & $13,\ldots,16$ & \cite{HO03}  \\
\hline
$\FF_7$ & $1,\ldots,9$   & \cite{PT}    \\
        & $10,\ldots,13$ & \cite{HO02}  \\
\noalign{\hrule height0.8pt}
    \end{tabular}
}
\end{center}
\end{table}

\subsection{Ternary maximal self-orthogonal codes}
\label{Sec:MSO}
An $\FF_3$-code is called ternary.
All ternary maximal self-orthogonal codes of lengths 
$4m+1,4m+2,4m+3$ can be obtained from self-dual codes
of length $4m+4$ by subtracting (see~\cite{CPS}).
A classification of ternary maximal self-orthogonal 
codes of lengths $3, \ldots,12$, lengths $13,14,15,16$ 
and lengths $17,18,19,20$
was done in~\cite{MPS}, \cite{CPS} and \cite{PSW}, respectively.

In the course of reproducing
a classification of ternary maximal self-orthogonal 
codes of lengths up to $20$,
we discovered errors in the classification 
for lengths $18$ and $19$.
The numbers of ternary maximal self-orthogonal codes 
of lengths $18$ and $19$ are listed 
in~\cite[Table IV]{PSW} as $154$ and $54$, respectively.
However, we verified that the correct numbers are
$160$ and $56$, respectively.
Let $C_{20,i}$ denote the $i$-th self-dual code of length $20$ 
given in~\cite[Tables II and III]{PSW},
and let $n_{18}(i)$ and $n_{19}(i)$ denote the numbers of 
inequivalent maximal self-orthogonal codes 
of lengths $18$ and $19$, respectively, obtained from $C_{20,i}$
by subtracting.
Let $C_{20,i}^{(k)}$ denote the self-orthogonal
code of length $19$ obtained from $C_{20,i}$ by subtracting
the $k$-th coordinate.
Although the numbers $n_{19}(20)$ and $n_{19}(23)$
are listed as both $1$ in~\cite[Table IV]{PSW},
we verified that the codes $C_{20,20}^{(i)}$
$(i=1,\ldots,20)$ are equivalent to one of the two
inequivalent codes $C_{20,20}^{(1)},C_{20,20}^{(20)}$, and
the codes $C_{20,23}^{(i)}$
$(i=1,\ldots,20)$ are equivalent to one of the two inequivalent codes
$C_{20,23}^{(1)},C_{20,23}^{(20)}$.
In fact, these four codes have different automorphism groups,
of orders $32, 128, 576$ and $5184$, respectively.
Hence, we conclude that  $n_{19}(20)=n_{19}(23)=2$.
Since \cite[Table IV]{PSW} also 
contains incorrect values for $n_{18}(i)$,
we list their correct values in Table~\ref{Tab:ternary1819}.

In order to check that a classification is complete, 
in all of the classification results, 
we first verified by {\sc Magma} that all codes
are inequivalent.
This was done by the {\sc Magma} function {\tt IsIsomorphic}, 
as well as by checking that all codes
have different numbers $(B_0,B_1,\ldots,B_n)$, where 
$B_j$ is the number of distinct cosets of weight $j$.
Then we checked the mass formula, that is,
we computed the sum in 
\begin{equation}\label{eq:mf}
\sum_{C \in {\cal C}} \frac{2^{n} \cdot n!}{|\Aut(C)|},
\end{equation}
where ${\cal C}$ is the  set of 
inequivalent maximal self-orthogonal codes of length $n$
and we checked against the known formula for 
the number $N_0$ of distinct maximal self-orthogonal codes
of length $n$, which is given in~\cite[p.~650]{MPS}.
The automorphism group $\Aut(C)$ of $C$
is calculated by the {\sc Magma} function {\tt AutomorphismGroup}. 
Note that each summand in (\ref{eq:mf})
expresses the cardinality of the
equivalence class of a given a code $C$ and
the sum of all these cardinalities is equal to $N_0$.
The numbers $\#$ of all inequivalent 
maximal self-orthogonal codes of lengths up to $20$
are listed in Table~\ref{Tab:ternary},
and generator matrices of those codes
can be obtained electronically from~\cite{Data}.

\begin{prop}
Up to equivalence, 
there are $160$ and $56$ ternary maximal self-orthogonal codes 
of lengths $18$ and $19$, respectively.
\end{prop}

\begin{table}[thbp]
\caption{Ternary maximal self-orthogonal codes of lengths $18$ and $19$}
\label{Tab:ternary1819}
\begin{center}
{\small
\begin{tabular}{c|c|c||c|c|c||c|c|c}
\noalign{\hrule height0.8pt}
$i$ & $n_{18}(i)$ & $n_{19}(i)$ &
$i$ & $n_{18}(i)$ & $n_{19}(i)$ &
$i$ & $n_{18}(i)$ & $n_{19}(i)$ \\
\hline 
 1 &  2 & 1 &  9 &  5 & 2 & 17 & 16 & 5\\
 2 &  5 & 2 & 10 &  8 & 2 & 18 & 12 & 4\\
 3 &  4 & 2 & 11 &  4 & 2 & 19 &  3 & 1\\
 4 &  7 & 3 & 12 & 12 & 4 & 20 & 10 & 2\\
 5 &  7 & 3 & 13 &  8 & 3 & 21 &  4 & 1\\
 6 &  6 & 3 & 14 &  9 & 3 & 22 &  5 & 1\\
 7 &  4 & 2 & 15 & 10 & 3 & 23 &  5 & 2\\
 8 &  5 & 2 & 16 &  7 & 2 & 24 &  2 & 1\\
\hline
\multicolumn{5}{c|}{}&\multicolumn{2}{c|}{Total} & 160 & 56 \\
\noalign{\hrule height0.8pt}
   \end{tabular}
}
\end{center}
\end{table}
\begin{table}[thbp]
\caption{Ternary maximal self-orthogonal codes}
\label{Tab:ternary}
\begin{center}
{\small
\begin{tabular}{c|c|l||c|c|l}
\noalign{\hrule height0.8pt}
Length &  $\#$ & \multicolumn{1}{c||}{References} &
Length &  $\#$ & \multicolumn{1}{c}{References} \\
\hline 
 3 &   1 & \cite{MPS}  &12 &   3 & \cite{MPS}  \\
 4 &   1 & \cite{MPS}  &13 &   7 & \cite{CPS} \\
 5 &   1 & \cite{MPS}  &14 &  22 & \cite{CPS} \\
 6 &   2 & \cite{MPS}  &15 &  12 & \cite{CPS} \\
 7 &   1 & \cite{MPS}  &16 &   7 & \cite{CPS} \\
 8 &   1 & \cite{MPS}  &17 &  23 & \cite{PSW} \\
 9 &   2 & \cite{MPS}  &18 & 160 & Section~\ref{Sec:MSO}\\
10 &   5 & \cite{MPS}  &19 &  56 & Section~\ref{Sec:MSO}\\
11 &   3 & \cite{MPS}  &20 &  24 & \cite{PSW} \\
\noalign{\hrule height0.8pt}
   \end{tabular}
}
\end{center}
\end{table}

\section{Classification method}\label{Sec:Method}

When $n$ is odd, 
the existence of a weighing matrix of order $n$ and weight $k$ implies
that $k$ is a square and $(n-k)^2+(n-k)+1 \ge n$.
When $n \equiv 2 \pmod 4$,
the existence of a weighing matrix of order $n$ and weight $k$ implies
that $k$ is the sum of two squares and
$k \le n-1$~\cite{CRS86}.

For the remainder of this section, let
$W=(w_{ij})$ be a weighing matrix of order $n$ and weight $k$.
The number $\sum_{s=1}^n w_{is}^2 w_{js}^2$
is called the intersection number of $i$-th row $r_i$ and the
$j$-th row $r_j$  $(i\ne j)$.
The maximum number among intersection numbers for rows 
of $W$ and $W^T$ is called the intersection number of 
$W$~\cite{OM98}.
We say that $r_j$ intersects $r_i$ in $2 \ell$ places 
if the intersection number is $2 \ell$~\cite{CRS86}.
For a fixed row $r_i$,
let $x_{2 \ell}$ be the numbers of rows $r_j$
other than $r_i$ such that the intersection number of 
$r_i$ and $r_j$ is $2 \ell$.
The sequence $(x_0,x_2,\ldots,x_{2\lfloor n/2 \rfloor})$ is called
the intersection pattern corresponding to $r_i$~\cite{CRS86}.
The number $\sum_{j=1}^n w_{sj}^2 w_{tj}^2 w_{uj}^2$
is called the generalized intersection number and 
the following set of generalized intersection numbers
\[
N(i)= \Big| \Big\{ \{s,t,u\} \mid 
\sum_{j=1}^n w_{sj}^2 w_{tj}^2 w_{uj}^2 =i,
1 \le s,t,u \le n\ (s \not=t, s \not=u, t \not=u) \Big\} \Big|
\]
is called the $g$-distribution (see~\cite{OM98}).
Note that there are inequivalent weighing matrices
with the same  $g$-distribution.

Let $C_m(W)$ be the $\ZZ_m$-code generated by the rows of $W$, where 
the entries of $W$ are regarded as elements of $\ZZ_m$.
The following is trivial.

\begin{lem}\label{lem:SO}
If $k$ is divisible by $m$, then $C_m(W)$ is self-orthogonal.
\end{lem}

\begin{prop}\label{prop:SD}
Let $p$ be an odd prime.
If $k$ is divisible by $p$ but $k$ is not divisible by $p^2$, then
$C_p(W)$ is a self-dual $\FF_p$-code.
\end{prop}
\begin{proof}
Suppose that $k=pt$, where $t$ is not divisible by $p$.
Since
\[
\det(W^2)=\det(W W^T)= \det(kI_n)= k^n,
\]
we have $|\det(W)|=k^{n/2}$.
Let $d_1|d_2|\cdots|d_n$ be the elementary divisors of $W$
(see e.g.~\cite[II.17]{Newman} for the definition of elementary divisors).
Then
\[
|\det(W)|= d_1d_2\cdots d_n =k^{\frac{n}{2}}=p^{\frac{n}{2}}t^{\frac{n}{2}}.
\]
Since $t$ is not divisible by $p$,
$n$ must be even, and at most $n/2$ $d_i$'s are divisible by $p$.
Hence, $\dim C_p(W) \ge n/2$.
By Lemma~\ref{lem:SO},  $\dim C_p(W) \le n/2$.
The result follows.
\end{proof}



From now on,
suppose that $\ZZ_m$ is either $\FF_p$ or $\ZZ_4$.
Let $n_i(x)$ denote the number of components $i$ of $x \in \ZZ_m^n$
($i \in \ZZ_m$).
Any row of $W$ is a codeword $x$ of $C_m(W)$
such that $n_0(x)=n-k$ and $n_1(x)+n_{-1}(x)=k$.
By Lemma~\ref{lem:SO}, $C_m(W)$ is self-orthogonal.
It follows that the rows of $W$
are composed of
$n$ codewords $x$ with
$n_0(x)=n-k$ and $n_1(x)+n_{-1}(x)=k$ in some
maximal self-orthogonal $\ZZ_m$-code of length $n$.

We now describe how all weighing matrices 
of order $n$ and weight $k=mt$ can be constructed from
maximal self-orthogonal $\ZZ_m$-codes of 
length $n$.
Let $C$ be a maximal self-orthogonal $\ZZ_m$-code of 
length $n$, and let $V$ be the set of pairs
$\{x,-x\}$ satisfying the condition that
$n_0(x)=n-k$, $n_1(x)+n_{-1}(x)=k$, $x \in C$.
We define the simple undirected graph $\Gamma$, whose set of  
vertices is the set $V$ and two vertices $\{x,-x\},\{y,-y\}\in V$  
are adjacent if $\overline{x}\:{\overline{y}}^T=0$, where 
$\overline{x}=(\overline{x_1},\ldots,\overline{x_n}) \in \{0,1,-1\}^n
\subset \ZZ^n$
is the vector with $\overline{x}\bmod{m}=x$.

It follows that the $n$-cliques in the graph $\Gamma$
are precisely the set of weighing matrices which generate 
subcodes of $C$.
It is clear that the group $\Aut(C)$ acts on the graph 
$\Gamma$ as an automorphism group, and
therefore, the classification of such weighing matrices
reduces to finding a set of representatives of
$n$-cliques of $\Gamma$ up to the action of $\Aut(C)$. 
This computation was performed in {\sc Magma}~\cite{Magma}, 
the results were then converted 
to weighing matrices.
In this way, by considering all inequivalent
maximal self-orthogonal $\ZZ_m$-codes of length $n$,
we obtain a set of weighing matrices which contain a representative
of every equivalence class of weighing matrices of order
$n$ and weight $k=mt$. 

Since a weighing matrix does not, in general generate a
maximal self-orthogonal code, two equivalent weighing matrices
may be contained in two inequivalent maximal self-orthogonal codes.
One could consider not only maximal but also all self-orthogonal 
codes, and then list only those weighing matrices which generate
the given code. This will avoid duplication of equivalent
weighing matrices in the classification. However, we took
a different approach for efficiency. Once we have a set of
weighing matrices which could possibly contain equivalent pairs
of weighing matrices, we perform equivalent testing by considering
the associated incidence structures. This construction 
of incidence structures is
given by~\cite[Theorem 6.8]{J}, and in our case, it is as follows.
Given a weighing matrix $W$ of order $n$, replacing 
$0,1,-1$ in each entry by the matrices
\[
\begin{pmatrix} 0&0\\0&0\end{pmatrix},
\begin{pmatrix} 1&0\\0&1\end{pmatrix},
\begin{pmatrix} 0&1\\1&0\end{pmatrix},
\]
respectively, we obtain a $(0,1)$-matrix of order $2n$.
This matrix defines a square incidence structure $\cD(W)$
with $2n$ points and $2n$ blocks. We may take the set of
points of $\cD(W)$ to be $\cP=\{\pm1,\pm2,\dots,\pm n\}$,
so that the permutation $\tau=(1,-1)(2,-2)\cdots(n,-n)$
is a fixed-point-free involutive automorphism of $\cD(W)$.
More precisely, the set of blocks $\cB(W)$ of $\cD(W)$ is
\[
\cB(W)=\{B_i^\varepsilon\mid 1\leq i\leq n,\;\varepsilon=\pm1\},
\]
where
\[
B_i^\varepsilon=\{\varepsilon w_{ij}j\mid 1\leq j\leq n,\;w_{ij}\neq0\}.
\]
Here an automorphism of  $\cD(W)$ is a permutation of $\cP$ 
which maps $\cB(W)$ to $\cB(W)$.
The set of all automorphisms
is called the automorphism group and is denoted by $\Aut(\cD(W))$.
If we denote the orbits on $\cP$ under $\tau$ by $P_1,\dots,P_n$,
then the following conditions hold.
\begin{itemize}
\item[\rm (i)]
$|B \cap P_i| \le 1$ for any $i$ $(1 \le i \le n)$ and any
block $B \in \cB(W)$,
\item[\rm (ii)]
for any two blocks $B,B'\in\cB(W)$ such that $B' \ne B, B^\tau$,
\[
|\{i\mid B\cap P_i=B'\cap P_i\neq\emptyset\}|
=|\{i\mid \emptyset\neq B\cap P_i\neq B'\cap P_i\neq\emptyset\}|.
\]
\end{itemize}
Let $W_1$ and $W_2$ be weighing matrices of the
same order and weight.
We say that $\cD(W_1)$ and $\cD(W_2)$ are equivalent if
there is a permutation  $\sigma$ of
$\cP$ which maps $\cB(W_1)$ to
$\cB(W_2)$. 
Obviously, the equivalence of $W_1$ and $W_2$
implies that of $\cD(W_1)$ and $\cD(W_2)$.
Conversely, 
the following lemma gives a criterion
under which the equivalence of $\cD(W_1)$ and $\cD(W_2)$
implies that of $W_1$ and $W_2$.


\begin{lem}\label{lem:equiv}
Let $W$ be a weighing matrix of order $n$, and let
$\cD(W)$ be the square incidence structure defined by $W$. 
Suppose that 
\[
\tau=\tau_0=(1,-1)(2,-2) \cdots (n,-n)
\]
is 
the unique fixed-point-free involutive automorphism
of $\cD(W)$ satisfying the conditions (i) and (ii) above,
up to conjugacy in $\Aut(\cD(W))$.
If $U$ is a weighing matrix such that
$\cD(U)$ is equivalent to $\cD(W)$,
then $U$ is equivalent to $W$.
\end{lem}
\begin{proof}
Let $\sigma$ denote a map from $\cD(U)$ to $\cD(W)$ giving an  
equivalence.
This means that $\sigma$ is a permutation of
$\cP$ which maps $\cB(U)$ to
$\cB(W)$. 

We first claim that $\tau=\sigma^{-1}\tau_0\sigma$
satisfies the conditions (i) and (ii) above. Indeed, 
the orbits on $\cP$ under $\tau$ are
$P_i=\{i,-i\}^\sigma$ $(1\leq i\leq n)$.
If $B\in\cB(W)$, then $B^{\sigma^{-1}}\in\cB(U)$, hence
$|B\cap P_i|=|B^{\sigma^{-1}}\cap\{i,-i\}|\leq1$.
Thus, (i) holds.
If $B,B'\in\cB(W)$ and $B'\neq B,B^\tau$, then
${B'}^{\sigma^{-1}}\neq B^{\sigma^{-1}}$ and
${B'}^{\sigma^{-1}}\neq B^{\tau\sigma^{-1}}
=B^{\sigma^{-1}\tau_0}$. Since (ii) holds for
$\cB(U)$ and $\tau_0$, we have
\begin{align*}
&|\{i\mid B^{\sigma^{-1}}\cap\{i,-i\}=
{B'}^{\sigma^{-1}}\cap \{i,-i\}\neq\emptyset\}|
\\ &=
|\{i\mid \emptyset\neq B^{\sigma^{-1}}\cap\{i,-i\}\neq
{B'}^{\sigma^{-1}}\cap \{i,-i\}\neq\emptyset\}|.
\end{align*}
Thus, (ii) holds. 
Therefore, the claim is proved. 

By assumption, then, $\tau$ is conjugate to $\tau_0$
in $\Aut(\cD(W))$. This implies that there exists an
automorphism $\pi\in\Aut(\cD(W))$ such that
$\sigma^{-1}\tau_0\sigma=\pi^{-1}\tau_0\pi$. Replacing
$\sigma$ by $\sigma\pi^{-1}$, we may assume from the
beginning that $\sigma$ commutes with $\tau_0$. Then
there exists a permutation $\rho\in S_n$ and $q_j\in\{\pm1\}$
such that
$(\pm j)^\sigma=\pm q_j j^\rho$.
Let
\begin{align*}
\cB(W)&=\{B_i^\varepsilon\mid 1\leq i\leq n,\;\varepsilon=\pm1\},\\
\cB(U)&=\{C_i^\varepsilon\mid 1\leq i\leq n,\;\varepsilon=\pm1\},
\end{align*}
where
\begin{align*}
B_i^\varepsilon&=\{\varepsilon w_{ij}j\mid 1\leq j\leq n,\;w_{ij}\neq0\},\\
C_i^\varepsilon&=\{\varepsilon u_{ij}j\mid 1\leq j\leq n,\;u_{ij}\neq0\}.
\end{align*}
Since $\cB(U)^\sigma=\cB(W)$,
for any $i$, there exists $i'$ and and $p_i\in\{\pm1\}$ such
that $(C_i^+)^\sigma=B_{i'}^{p_i}$. 
Since $\sigma$ commutes with $\tau_0$, we have
$(C_i^-)^\sigma
=B_{i'}^{-p_i}$.
This implies that 
there exists a permutation $\pi\in S_n$
such that
$(C_i^+)^\sigma=B_{i^\pi}^{p_i}$.
Thus, $q_j u_{ij}=p_i w_{i^\pi,j^\rho}$.

Now, define monomial matrices
$P=(p_i\delta_{i^\pi,j})$,
$Q=(q_i\delta_{i^\rho,j})$.
Then we obtain $PWQ^{-1}=U$.
Therefore, $W$ is equivalent to $U$.
\end{proof}

\section{Weighing matrices of weight 5}\label{Sec:wt5}

In the course of reproducing
a classification of weighing matrices of 
order $12$ (see Section~\ref{Sec:12}),
we discovered errors in the classification 
of weighing matrices of weight $5$ given in~\cite[Theorem 5]{CRS86}.
In this section, we give a revised classification of 
weighing matrices of weight $5$.

In the proof of~\cite[Theorem 5]{CRS86}, the authors of~\cite{CRS86}
divide the classification into the following three cases:
\begin{itemize}
\item[(a)]
at least two other rows intersect the first row in four places or,
\item[(b)]
no rows intersect any other row in four places or,
\item[(c)]
exactly one row intersects the first row in four places.
\end{itemize}
Then all weighing matrices of weight $5$ for the three
cases (a), (b) and (c) were classified in~\cite[Theorem 5]{CRS86}.
In the proof of~\cite[Theorem 5]{CRS86},
$D(16,5)$ is claimed to be the unique weighing matrix of weight $5$  
satisfying (b).
However, we found more weighing matrices of weight $5$ satisfying
(b).
In Figure~\ref{Fig:1214-5},
we give such weighing matrices $W_{12,5}$ and $W_{14,5}$
of orders $12$ and $14$, respectively.

\begin{figure}[t]
\begin{center}
{\footnotesize
\begin{align*}
W_{12,5}&=
\left(\begin{array}{rrrrrrrrrrrr}
  1&  1&  1&  1&  1&  0&  0&  0&  0&  0&  0&  0\\
  1& -1&  0&  0&  0&  1&  1&  1&  0&  0&  0&  0\\
  1&  0& -1&  0&  0& -1&  0&  0&  1&  1&  0&  0\\
  1&  0&  0& -1&  0&  0& -1&  0& -1&  0&  1&  0\\
  1&  0&  0&  0& -1&  0&  0& -1&  0& -1& -1&  0\\
  0&  1& -1&  0&  0&  0&  0&  1&  0& -1&  0&  1\\
  0&  1&  0& -1&  0&  1&  0&  0&  0&  1& -1&  0\\
  0&  1&  0&  0& -1&  0&  1&  0&  0&  0&  1& -1\\
  0&  0&  1& -1&  0& -1&  1&  0&  0&  0&  0&  1\\
  0&  0&  1&  0& -1&  0& -1&  1&  1&  0&  0&  0\\
  0&  0&  0&  1& -1&  0&  0&  0& -1&  1&  0&  1\\
  0&  0&  0&  0&  0&  1&  0& -1&  1&  0&  1&  1
\end{array}\right)
\\
W_{14,5}&=
\left(\begin{array}{rrrrrrrrrrrrrr}
  1&  1&  1&  1&  1&  0&  0&  0&  0&  0&  0&  0&  0&  0\\
  1& -1&  0&  0&  0&  1&  1&  1&  0&  0&  0&  0&  0&  0\\
  1&  0& -1&  0&  0& -1&  0&  0&  1&  1&  0&  0&  0&  0\\
  1&  0&  0& -1&  0&  0& -1&  0& -1&  0&  1&  0&  0&  0\\
  1&  0&  0&  0& -1&  0&  0& -1&  0& -1& -1&  0&  0&  0\\
  0&  1& -1&  0&  0&  1&  0&  0&  0&  0&  0&  1&  1&  0\\
  0&  1&  0& -1&  0&  0&  0&  1&  0&  0& -1&  0& -1&  0\\
  0&  1&  0&  0& -1&  0&  1&  0&  0&  0&  1& -1&  0&  0\\
  0&  0&  1& -1&  0&  0&  0&  0&  1&  0&  0&  0&  1& -1\\
  0&  0&  1&  0& -1&  0&  0&  0&  0&  1&  0&  1&  0&  1\\
  0&  0&  0&  1& -1&  0& -1&  1&  0&  0&  0&  0&  0& -1\\
  0&  0&  0&  0&  0&  1& -1&  0&  1&  0&  0& -1&  0&  1\\
  0&  0&  0&  0&  0&  1&  0& -1&  0&  1&  0&  0& -1& -1\\
  0&  0&  0&  0&  0&  0&  0&  0&  1& -1&  1&  1& -1&  0
\end{array}\right)
\end{align*}
\caption{Weighing matrices of orders $12,14$ and weight $5$}
\label{Fig:1214-5}
}
\end{center}
\end{figure}

\begin{lem}\label{lem:wt5}
Let $W$ be a weighing matrix of order $2n$ and weight $5$ 
satisfying the condition (b). 
Then $W$ contains $W_{12,5}$, $W_{14,5}$ or
$D(16,5)$ as a direct summand.
\end{lem}
\begin{proof}
From the condition (b),
we may assume without loss of generality that 
the first $5$ rows of $W$ have the following form:
\[
M_1=
\left(\begin{array}{cc|ccc|ccc|ccc|c}
+&+&+&+&+&0&0&0&0&0&0&0\cdots0\\
+&-&0&0&0&+&+&+&0&0&0&0\cdots0\\
+&0&-&0&0&-&0&0&+&+&0&0\cdots0\\
+&0&0&-&0&0&-&0&-&0&+&0\cdots0\\
+&0&0&0&-&0&0&-&0&-&-&0\cdots0
\end{array}\right),
\]
where $+,-$ denote $1,-1$, respectively.
In addition, we may assume without loss of generality that 
the next three rows have the following form:
\[
M_2=
\left(\begin{array}{cc|ccc|ccc|ccc|ccc}
0 & + & - & 0 & 0 &  &  & & &  &  &  &    & \\ 
0 & + & 0 & - & 0 &  & A& & & B&  &  & C  & \\
0 & + & 0 & 0 & - &  &  & & &  &  &  &    & \\
\end{array}\right),
\]
where $A$ is a $3 \times 3$ permutation matrix,
$B$ is some $3 \times 3$ matrix and 
$C$ is some $3 \times (2n-11)$ matrix.
Let $M(A,B,C)$ denote the matrix
\[
\left(\begin{array}{c}
M_1 \\
M_2
\end{array}\right).
\]
If $A'=P_1AP_1^{-1}$ for some $3 \times 3$ permutation
matrix $P_1$,
then 
\[
PM(A,B,C)P^{-1}=M(A',B,C),
\]
where 
\[
P=
\left(\begin{array}{cccc}
I_2 &    &    &   \\
    & P_1&    &   \\
    &    & P_1&   \\
    &    &    & I_{2n-8}
\end{array}\right).
\]
This means that it is sufficient to consider the matrix $A$
up to conjugacy in the symmetric group of degree $3$,
so we assume $A=I_3,A_2$ or $A_3$, where
\[
A_2=
\left(\begin{array}{ccc}
+ & 0 & 0\\
0 & 0 & +\\
0 & + & 0
\end{array}\right) \text{ and }
A_3=
\left(\begin{array}{ccc}
0 & 0 & +\\
+ & 0 & 0\\
0 & + & 0
\end{array}\right).
\]

\begin{itemize}
\item Case $A=I_3$:\\
From the orthogonality of rows,
\[
B=
\left(\begin{array}{ccc}
0 & 0 & 0\\
0 & 0 & 0\\
0 & 0 & 0
\end{array}\right)
\text{ and } 
C=
\left(\begin{array}{ccc|c}
 + & + & 0 &0\cdots0\\
 - & 0 & - &0\cdots0\\
 0 & - & + &0\cdots0
\end{array}\right).
\]
Moreover, the matrix $M(A,B,C)$ is uniquely extended to
\[
W=
\left(\begin{array}{ccc|c}
  &          & &  \\
  &  D(16,5) & & O\\
  &          & &  \\
\hline
  &   O      & & * \\
\end{array}\right),
\]
up to equivalence, where $O$ is the zero matrix.

\item Case $A=A_2$: \\
From the orthogonality of rows,
\[
B=
\left(\begin{array}{ccc}
0 & 0 & 0\\
0 & 0 & -\\
0 & 0 & +
\end{array}\right) \text{ and }
C=
\left(\begin{array}{cc|c}
 + & + & 0\cdots0\\
 0 & - & 0\cdots0\\
 - & 0 & 0\cdots0
\end{array}\right).
\]
Moreover, the matrix $M(A,B,C)$ is uniquely extended to
\[
W=
\left(\begin{array}{ccc|c}
  &          & &  \\
  & W_{14,5} & & O\\
  &          & &  \\
\hline
  &   O      & & * \\
\end{array}\right),
\]
up to equivalence, where $W_{14,5}$ is given in
Figure~\ref{Fig:1214-5}.

\item Case $A=A_3$: \\
From the orthogonality of rows, $B$ must be one of 
the following three matrices:
\[
\left(\begin{array}{ccc}
 0 & - & 0\\
 0 & + & -\\
 0 & 0 & +
\end{array}\right),
\left(\begin{array}{ccc}
 - & 0 & -\\
 + & 0 & 0\\
 0 & 0 & +
\end{array}\right) 
\text{ and }
\left(\begin{array}{ccc}
 0 & - & 0\\
 + & 0 & 0\\
 - & + & 0
\end{array}\right).
\]
Then $C$ can be considered as  
\[
\left(\begin{array}{c|c}
  +& 0\cdots0\\
  0& 0\cdots0\\
  -& 0\cdots0
\end{array}\right),
\left(\begin{array}{c|c}
  0& 0\cdots0\\
  +& 0\cdots0\\
  -& 0\cdots0
\end{array}\right) 
\text{ and }
\left(\begin{array}{c|c}
  +& 0\cdots0\\
  -& 0\cdots0\\
  0& 0\cdots0
\end{array}\right), 
\]
respectively.
Moreover, for each case
the matrix $M(A,B,C)$ is uniquely extended to
\[
W=
\left(\begin{array}{ccc|c}
  &          & &  \\
  & W_{12,5} & & O\\
  &          & &  \\
\hline
  &   O      & & * \\
\end{array}\right),
\]
up to equivalence, where $W_{12,5}$ is given in
Figure~\ref{Fig:1214-5}.
\end{itemize}
Therefore, $W$ contains $W_{12,5}$, $W_{14,5}$ or
$D(16,5)$ as a direct summand.
\end{proof}

\begin{rem}
For order $14$, it follows from~\cite[Theorem 5]{CRS86}
that there are two inequivalent weighing matrices of
weight $5$, namely, $E(14,5)$ and $W(6,5)\oplus W(8,5)$
in~\cite{CRS86}.
On the other hand, the table in~\cite[Appendix B]{CRS86}
lists the number of inequivalent weighing matrices of
weight $5$ to be three, and the missing matrix is
denoted by $D(14,5)$ which, however, is not defined in~\cite{CRS86}.
\end{rem}

\begin{rem} 
Let $R=(r_{ij})$ be the square matrix of order
$n$ with $r_{ij}=1$ if $i+j-1=n$ and $0$ otherwise.
If $A_1$ and $A_2$ are circulant matrices of order $n$
with entries $0,\pm1$ satisfying $A_1A_1^T+A_2A_2^T=kI$, then
the matrices
\[
W_1=\left(\begin{array}{ccc}
A_1 & A_2 \\
-A_2^T & A_1^T
\end{array}\right)
\text{ and }
W_2=\left(\begin{array}{ccc}
A_1 & A_2R \\
-A_2R & A_1
\end{array}\right)
\]
are weighing matrices of order $2n$ and weight 
$k$~\cite[Proposition~4.46]{GS}.
Kotsireas and Koukouvinos~\cite{KK09} claim that all
weighing matrices of the form $W_1$ or $W_2$
are found by an exhaustive search for $n \le 11$.
Although the results of their search are not given,
this means that they must have found the weighing matrix
$W_{14,5}$, since it is equivalent to the weighing matrix
$W_1$ where $A_1$ and $A_2$ are the circulant matrices
with first rows
\[
( 1,  0,  0,  0,  0,  0,  0) \text{ and }
(-1,  1,  1,  0,  1,  0,  0),
\]
respectively. We verified that no weighing matrices $W_1,W_2$
constructed from two circulant matrices $A_1$ and $A_2$
are equivalent to $W_{12,5}$.
This was done by finding all weighing matrices of the form
$W_1$ and $W_2$ by an exhaustive search.
\end{rem}

\begin{rem}
Let $W$ be any of $W_{12,5}, W_{14,5}$ and $D(16,5)$.
Then $W^T$ also satisfies (b).
Let $\overline{W}$ be the $(1,0)$-matrix obtained from $W$
by changing $-1$ to $1$ in the entries.
Then $\overline{W}$ is the incidence matrix of a semibiplane
(see~\cite{Wild} for the definition of a semibiplane).
The three semibiplanes obtained in this way are given 
in~\cite[Proposition 15]{Wild}.
\end{rem}

By the above lemma, we have the following revised classification for
weight $5$.
See~\cite{CRS86} for the definitions of the weighing matrices
$W(6,5)$, $W(8,5)$, $E(4t_i+2,5)$ and  
$F(4t_j+4,5)$.

\begin{thm}\label{thm:wt5}
Any weighing matrix of order $2n$ and weight $5$ is
equivalent to
\begin{multline*}
\bigoplus_{i_1} W(6,5) 
\bigoplus_{i_2} W(8,5) 
\bigoplus_{i_3} W_{12,5} 
\bigoplus_{i_4} W_{14,5} 
\\
\bigoplus_{i_5} D(16,5)
\bigoplus_{t_i} \big(\bigoplus_{i_6} E(4t_i+2,5)\big)
\bigoplus_{t_j} \big(\bigoplus_{i_8} F(4t_j+4,5)\big),
\end{multline*}
where $t_i \ge 2$ and $t_j \ge 2$.
\end{thm}

Table~\ref{Tab:W5} is a revised table of a classification
of weighing matrices of order $2n \le 20$
and weight $5$ in~\cite[Appendix B]{CRS86}.

\begin{table}[thb]
\caption{Weighing matrices of weight $5$}
\label{Tab:W5}
\begin{center}
{\footnotesize
\begin{tabular}{c|c|l}
\noalign{\hrule height0.8pt}
$2n$ & $\#$ & \multicolumn{1}{c}{Matrices} \\
\hline 
 6& 1& $W(6,5)$  \\
 8& 1& $W(8,5)$  \\
10& 1& $E(10,5)$ \\
12& 3& $W_{12,5}$, $F(12,5)$, $W(6,5)\oplus W(6,5)$ \\
14& 3& $W_{14,5}$, $E(14,5)$, $W(6,5)\oplus W(8,5)$ \\ 
16& 4& $D(16,5)$, $F(16,5)$, $W(8,5)\oplus W(8,5)$, $W(6,5)\oplus E(10,5)$\\
18& 5& $E(18,5)$, $W(6,5)\oplus W_{12,5}$, $W(6,5)\oplus F(12,5)$,
$W(6,5)\oplus W(6,5)\oplus W(6,5)$\\ 
&& $W(8,5)\oplus F(10,5)$\\
20& 7& $F(20,5)$, $W(6,5) \oplus W_{14,5}$, $W(6,5) \oplus E(14,5)$, 
       $W(6,5) \oplus W(6,5)\oplus W(8,5)$ \\
&& $W(8,5) \oplus W_{12,5}$, $W(8,5) \oplus F(12,5)$, 
$E(10,5) \oplus E(10,5)$\\
\noalign{\hrule height0.8pt}
   \end{tabular}
}
\end{center}
\end{table}

\section{Weighing matrices of order 12}\label{Sec:12}
The classification of weighing matrices of order $12$
was done in~\cite{CRS86} and \cite{O89}.
In this section, we give a revised list of
weighing matrices of weights $6,8,10$.
These classifications were done by considering 
self-dual $\ZZ_k$-codes of length $12$, where 
$k=3,4$ and $5$, respectively, using the method in
Section~\ref{Sec:Method}.
These approaches are similar, and we give details only for
weight $6$.

\subsection{Weight 6}
As described in Section~\ref{Sec:Method},
any weighing matrix of order $12$ and weight $6$
can be regarded  as $12$ codewords of weight $6$ in 
some ternary self-dual code of length $12$.
There are three inequivalent ternary self-dual codes of
length $12$ \cite[Table~1]{MPS}, 
and these codes are denoted by $G_{12}$, $4C_3(12)$
and $3E_4$.
The code $G_{12}$ has minimum weight $6$ and the other codes
have minimum weight $3$, and 
the numbers of codewords of weight $6$ in 
these codes are $264$, $240$ and $192$, respectively.
By considering sets of $12$ codewords of weight $6$ in 
these codes, we have the following classification
of weighing matrices of
order $12$ and weight $6$, 
using the method in Section~\ref{Sec:Method}.

\begin{thm}\label{thm:12-6}
There are $8$ inequivalent weighing matrices of
order $12$ and weight $6$.
\end{thm}

The number of inequivalent weighing matrices of
order $12$ and weight $6$ was incorrectly reported as $7$ 
in~\cite[Theorem 5]{O89}.
The $7$ inequivalent matrices in~\cite[Theorem 5]{O89}
are denoted by 
$E^*_1$,
$E^*_2$,
$(E^*_2)^T$,
$E^*_5$,
$E^*_{14}$,
$(E^*_{14})^T$ and 
$G^*_2$.
The missing matrix $W_{12,6}$ is listed in
Figure~\ref{Fig:12}.
We remark that $W_{12,6}$ and $W_{12,6}^T$ are equivalent.

\begin{figure}[t]
\begin{center}
{\footnotesize
\begin{align*}
W_{12,6}&=
\left(\begin{array}{rrrrrrrrrrrr}
 1&  1&  1&  1&  1&  1&  0&  0&  0&  0&  0&  0\\
 1&  1& -1&  0& -1&  0&  0&  1&  1&  0&  0&  0\\
 1& -1&  0&  1& -1&  0&  0& -1&  0&  0&  0&  1\\
 1& -1&  0&  0&  0&  0&  0&  1& -1&  0&  1& -1\\
 1&  0&  0& -1&  1& -1&  0&  0&  0&  1&  0&  1\\
 1&  0&  0& -1&  0&  0&  0& -1&  0& -1& -1& -1\\
 0&  1&  1&  0& -1& -1&  1&  0& -1&  0&  0&  0\\
 0&  1& -1&  0&  0&  0& -1& -1& -1&  0&  1&  0\\
 0&  0&  1& -1& -1&  1& -1&  0&  0&  1&  0&  0\\
 0&  0&  1&  0&  0& -1& -1&  0&  1& -1&  1&  0\\
 0&  0&  0&  1&  0& -1& -1&  0&  0&  1& -1& -1\\
 0&  0&  0&  0&  0&  0&  1& -1&  1&  1&  1& -1\\
\end{array}\right)
\\
W_{12,10}&=
\left(\begin{array}{rrrrrrrrrrrr}
 1&  1&  1&  1&  1&  1&  1&  1&  1&  1&  0&  0\\
 1&  1&  1&  1&  1& -1& -1& -1& -1& -1&  0&  0\\
 0&  1&  1& -1& -1&  0&  1& -1&  1& -1&  1&  1\\
-1&  0&  1&  1& -1& -1&  0&  1& -1&  1&  1&  1\\
-1& -1&  0&  1&  1&  1& -1&  0&  1& -1&  1&  1\\
 1& -1& -1&  0&  1& -1&  1& -1&  0&  1&  1&  1\\
 1&  1& -1& -1&  0&  1& -1&  1& -1&  0&  1&  1\\
 0&  1& -1&  1& -1&  0& -1& -1&  1&  1& -1&  1\\
-1&  0&  1& -1&  1&  1&  0& -1& -1&  1& -1&  1\\
 1& -1&  0&  1& -1&  1&  1&  0& -1& -1& -1&  1\\
-1&  1& -1&  0&  1& -1&  1&  1&  0& -1& -1&  1\\
 1& -1&  1& -1&  0& -1& -1&  1&  1&  0& -1&  1\\
\end{array}\right)
\end{align*}
\caption{Weighing matrices $W_{12,6}$ and $W_{12,10}$}
\label{Fig:12}
}
\end{center}
\end{figure}

\begin{rem}
It is claimed in the proof of~\cite[Lemma 31]{O89} that
there are $4$ weighing matrices which are constructed from
Case II up to equivalence.
The matrix $W_{12,6}$ is also constructed from Case II,
and hence there are $5$ weighing matrices which are constructed from
Case II up to equivalence. 
\end{rem}


In Table~\ref{Tab:12-6}, we list $g$-distributions
$N(i)$ $(i=0,1,\ldots,6)$ for the $7$
matrices given in~\cite[Theorem 5]{O89}
along with the new matrix $W_{12,6}$.
Table~\ref{Tab:12-6} also shows that the 
$8$ weighing matrices are inequivalent.
By Proposition~\ref{prop:SD},
the ternary codes $C_3(W)$ generated by the rows of these
matrices $W$ are self-dual,
and the identifications with those
appearing in~\cite{MPS} are given in
the last column of Table~\ref{Tab:12-6}.

\begin{table}[thb]
\caption{Weighing matrices of order $12$ and weight $6$}
\label{Tab:12-6}
\begin{center}
{\small
\begin{tabular}{c|ccccccc|c}
\noalign{\hrule height0.8pt}
$W$ & $N(0)$& $N(1)$& $N(2)$& $N(3)$& $N(4)$& $N(5)$& $N(6)$ 
& $C_3(W)$\\
\hline 
$E^*_1$       &396&720&180& 240& 180& 0&  0& $G_{12}$ \\ 
$E^*_2$       &516&432&420& 144& 204& 0&  0& $G_{12}$ \\ 
$(E^*_2)^T$   &432&528&504&  48& 192& 0& 12& $4C_3(12)$ \\ 
$E^*_5$       &492&432&468& 144& 180& 0&  0& $G_{12}$ \\ 
$E^*_{14}$    &708&  0&756&   0& 252& 0&  0& $G_{12}$ \\ 
$(E^*_{14})^T$&384&576&576&   0& 144& 0& 36& $3E_4$ \\ 
$G^*_2$       &432&576&456&   0& 240& 0& 12& $4C_3(12)$ \\ 
$W_{12,6}$    &516&360&540& 120& 180& 0&  0& $G_{12}$ \\ 
\noalign{\hrule height0.8pt}
   \end{tabular}
}
\end{center}
\end{table}

\subsection{Weight 8}
According to~\cite[Theorem 3]{O89},
there are $6$ inequivalent weighing matrices of
order $12$ and weight $8$.
However,
our method in Section~\ref{Sec:Method},
applied to the known classification of self-dual 
$\ZZ_4$-codes of length $12$ given in \cite{FGLP-Z4},
leads to the following classification of weighing
matrices of order $12$ and weight $8$.


\begin{thm}\label{thm:12-8}
There are $7$ inequivalent weighing matrices of
order $12$ and weight $8$.
\end{thm}

\begin{rem}
The $6$ inequivalent matrices in~\cite[Theorem 3]{O89}
are denoted by $A_1,A_3,A_6,A_7,A_8$ and $A_9$.
However, $A_{11}$ appeared in the proof of~\cite[Theorem 3]{O89} 
is inequivalent to any of the matrices $A_i$
$(i=1,3,6,7,8,9)$.
This is an error in~\cite[Theorem 3]{O89}.
\end{rem}

Let $W$ be a weighing matrix of order $12$ and weight $8$.
Let $D_4(W)$ be the $\ZZ_4$-code with generator matrix
$(\ I_{12}\ ,\ W \ )$, where the matrix $W$ is regarded
as a matrix over $\ZZ_4$.
The numbers $\#D_6$ of codewords of weight $6$ of
$D_4(W)$, listed in Table~\ref{Tab:12-8},
were found by the {\sc Magma} function {\tt NumberOfWords}.
These numbers also show that the 
$7$ weighing matrices are inequivalent,

\begin{table}[thb]
\caption{Weighing matrices of order $12$ and weight $8$}
\label{Tab:12-8}
\begin{center}
{\small
\begin{tabular}{c|ccccccc}
\noalign{\hrule height0.8pt}
$W$    & $A_1$  & $A_3$ & $A_6$ & $A_7$ & $A_8$ & $A_9$ & $A_{11}$  \\
\hline
$\#D_6$&  2852  &  1764 & 1092  &  932  & 1124  &  1700 & 1220 \\
\noalign{\hrule height0.8pt}
   \end{tabular}
}
\end{center}
\end{table}

\subsection{Weight 10}
According to~\cite[Theorem 1]{O89},
there are $4$ inequivalent weighing matrices of
order $12$ and weight $10$.
However, 
our method in Section~\ref{Sec:Method},
applied to the known classification of self-dual 
$\FF_5$-codes of length $12$ given in \cite{LPS}, 
leads to the following classification of weighing
matrices of order $12$ and weight $10$.

\begin{thm}\label{thm:12-10}
There are $5$ inequivalent weighing matrices of
order $12$ and weight $10$.
\end{thm}

The $4$ inequivalent matrices in~\cite[Theorem 1]{O89}
are denoted by $A_1,A_4,A_7$ and $A_8$.
The missing matrix $W_{12,10}$ is listed in
Figure~\ref{Fig:12}.
We remark that $W_{12,10}$ and $W_{12,10}^T$ are equivalent.

\begin{rem}
It is claimed in the proof of~\cite[Lemma 11]{O89} that
there are only $7$ vectors such that the matrices $Y_i$
are normal matrices of level $4$.
We verified that this is incorrect and there is one missing
vector, namely, the fourth row of $W_{12,10}$.
Moreover, the $7 \times 12$ matrix $\bar{Y}_4$
consisting of the first $7$ rows of $W_{12,10}$ should be
considered in~\cite[Lemma 12]{O89}
as a possible matrix of level $7$.
\end{rem}

In Table~\ref{Tab:12-10}, we list 
the self-dual $\FF_5$-codes $C_5(W)$ generated by the rows of 
these matrices $W$,
in the notation of~\cite{LPS}.
This shows that $W_{12,6}$ must be inequivalent to
any of the other $4$ matrices.
We consider $\FF_5$-codes $D_5(W)$
with generator matrices 
$(\ I_{12}\ ,\ W \ )$, where the matrices $W$ are regarded
as matrices over $\FF_5$.
The numbers $\#D_8$ of codewords of weight $8$ 
are listed in Table~\ref{Tab:12-10},
which also shows that the $5$ weighing matrices are inequivalent.
These numbers 
were found by the {\sc Magma} function {\tt NumberOfWords}.

\begin{table}[thb]
\caption{Weighing matrices of order $12$ and weight $10$}
\label{Tab:12-10}
\begin{center}
{\small
\begin{tabular}{c|c|c}
\noalign{\hrule height0.8pt}
$W$ & $C_5(W)$ &  $\#D_8$ \\ 
\hline 
$A_1$       & $F_6^2$ 
& 3696\\
$A_4$       & $F_{12}$ 
& 3000\\
$A_7$       & $F_{12}$ 
& 4080\\
$A_8$       & $F_6^2$ 
& 4560\\
$W_{12,10}$ & $K_{12}$ 
& 3792\\
\noalign{\hrule height0.8pt}
   \end{tabular}
}
\end{center}
\end{table}

\subsection{Other weights}
By Theorem~\ref{thm:wt5} (see Table~\ref{Tab:W5}),
there are $3$ inequivalent weighing matrices of
order $12$ and weight $5$, namely,
$W_{12,5}$, $F(12,5)$, and $W(6,5)\oplus W(6,5)$.
In Table~\ref{Tab:12-5}, we list 
the self-dual $\FF_5$-codes $C_5(W)$ generated by the rows of 
these matrices $W$,
in the notation of~\cite{LPS}.
\begin{table}[thb]
\caption{Weighing matrices of order $12$ and weight $5$}
\label{Tab:12-5}
\begin{center}
{\small
\begin{tabular}{c|ccc}
\noalign{\hrule height0.8pt}
$W$      & $W(6,5)\oplus W(6,5)$& $F(12,5)$             & $W_{12,5}$           \\
\hline
$C_5(W)$ & $F_6^2$ 
& $F_{12}$ 
& $K_{12}$ 
\\
\noalign{\hrule height0.8pt}
   \end{tabular}
}
\end{center}
\end{table}

For weights $7$ and $9$, we verified that
the classifications in~\cite{O89} are correct,
using the classification of 
self-dual $\FF_p$-codes of length $12$, 
where $p=7$ and $3$, respectively.
Table~\ref{Tab:12} summarizes a revised classification of
weighing matrices of order $12$.

\begin{table}[thb]
\caption{Classification of weighing matrices of order $12$}
\label{Tab:12}
\begin{center}
{\small
\begin{tabular}{c|c|l||c|c|l}
\noalign{\hrule height0.8pt}
Weight &  $\#$ & \multicolumn{1}{c||}{References} &
Weight &  $\#$ & \multicolumn{1}{c}{References} \\
\hline 
 1 & 1 & \cite{CRS86}         & 7 & 3 & \cite{O89} 
 \\
 2 & 1 & \cite{CRS86}         & 8 & 7 & Theorem~\ref{thm:12-8}
 \\
 3 & 1 & \cite{CRS86}         & 9 & 4 & \cite{O89} 
 \\
 4 & 5 & \cite{CRS86}         &10 & 5 & Theorem~\ref{thm:12-10}
 \\
 5 & 3 & Theorem~\ref{thm:wt5}&11 & 1 & \cite{CRS86} \\ 
 6 & 8 & Theorem~\ref{thm:12-6} &12 & 1 & \cite{Todd}  \\
\noalign{\hrule height0.8pt}
   \end{tabular}
}
\end{center}
\end{table}

\section{Weighing matrices of orders 14, 15 and 17}\label{Sec:141517}

We continue a classification of weighing matrices 
using the method in Section~\ref{Sec:Method}.
Then we have the following  classification of
weighing matrices of order $n$ and weight $k$
for 
\begin{equation}\label{eq:141517}
(n,k)=(14,8), (14,9), (14,10), (15,9) \text{ and } (17,9),
\end{equation}
using the classification of maximal
self-orthogonal $\ZZ_m$-codes of length $n$
(see Section~\ref{Sec:code}), where 
$m=4,3,5,3$ and $3$, respectively.
Since approaches are similar to that used in 
Section~\ref{Sec:12},
we only list in Table~\ref{Tab:141517} the 
numbers $\#$ of inequivalent 
weighing matrices of order $n$ and weight $k$
for $(n,k)$ listed in (\ref{eq:141517}).
Hence, our result completes a classification of 
weighing matrices of orders up to $15$ and order $17$.


\begin{table}[thb]
\caption{Classification of weighing matrices of orders $14,15$ and $17$}
\label{Tab:141517}
\begin{center}
{\small
\begin{tabular}{c|c|c|l||c|c|c|l}
\noalign{\hrule height0.8pt}
Order & Weight &  $\#$ & \multicolumn{1}{c||}{References} &
Order & Weight &  $\#$ & \multicolumn{1}{c}{References} \\
\hline 
14 & 1 &  1 & \cite{CRS86}       &15 & 1 & 1  & \cite{CRS86} \\
   & 2 &  1 & \cite{CRS86}       &   & 4 & 6  & \cite{CRS86} \\
   & 4 &  3 & \cite{CRS86}       &   & 9 & 37 & Section~\ref{Sec:141517}  \\
   & 5 &  3 & Theorem~\ref{thm:wt5}   &17 & 1 & 1    & \cite{CRS86} \\
   & 8 & 66 & Section~\ref{Sec:141517}&   & 4 & 3    & \cite{CRS86} \\
   & 9 &  7 & Section~\ref{Sec:141517}&   & 9 & 2360 & Section~\ref{Sec:141517}  \\
   &10 & 19 & Section~\ref{Sec:141517}&   &16 & 1    & \cite{CRS86} \\
   &13 &  1 & \cite{CRS86}            & &&\\
\noalign{\hrule height0.8pt}
   \end{tabular}
}
\end{center}
\end{table}

Now, we compare our classifications with the known classifications
for $(n,k)=(14,8)$ and $(17,9)$.
According to~\cite[Theorem 3.9]{O92},
there are $65$ inequivalent weighing matrices of 
order $14$ and weight $8$.
Using the classification of self-dual $\ZZ_4$-codes of length $14$,
we classified weighing matrices of order $14$ and weight $8$, and
we claim that the classification 
in~\cite[Theorem 3.9]{O92} misses the 
matrix $W_{14,8}$, which is listed in
Figure~\ref{Fig:148}.
We remark that $W_{14,8}$ and $W_{14,8}^T$ are equivalent.
Hence, we have the following:

\begin{thm}
There are $66$ inequivalent weighing matrices of order $14$ and 
weight $8$.
\end{thm}

\begin{figure}[t]
\begin{center}
{\footnotesize
\[
W_{14,8}=
\left(\begin{array}{rrrrrrrrrrrrrr}
1&  1&  1&  1&  1&  1&  1&  1&  0&  0&  0&  0&  0&  0\\
1& -1&  1&  1& -1& -1&  0&  0&  1&  1&  0&  0&  0&  0\\
1& -1&  0&  0&  0&  0&  1& -1& -1& -1&  1&  1&  0&  0\\
1&  1&  0& -1&  0&  0&  0& -1&  1&  0&  1& -1&  1&  0\\
1&  0& -1&  1&  0&  0& -1&  0&  0& -1& -1&  0&  1&  1\\
1&  0& -1&  0&  0&  1&  0& -1&  0&  1& -1&  0& -1& -1\\
1&  0&  0& -1&  1& -1& -1&  1&  0&  0&  0&  1&  0& -1\\
1&  0&  0& -1& -1&  0&  0&  1& -1&  0&  0& -1& -1&  1\\
0&  1&  1&  0&  0&  0& -1& -1& -1&  1&  0&  1&  0&  1\\
0&  1& -1&  0&  0& -1&  1&  0&  1&  0&  0&  1& -1&  1\\
0&  1&  0&  1& -1&  0& -1&  0&  0& -1&  1&  0& -1& -1\\
0&  1&  0&  0& -1& -1&  1&  0& -1&  0& -1&  0&  1& -1\\
0&  0&  1& -1& -1&  1&  0&  0&  1& -1& -1&  1&  0&  0\\
0&  0&  1&  0&  1& -1&  0& -1&  0& -1& -1& -1& -1&  0
\end{array}\right).
\]
\caption{Weighing matrix $W_{14,8}$}
\label{Fig:148}
}
\end{center}
\end{figure}

\begin{rem}
The intersection patterns of $W_{14,8}$ and
$W_{14,8}^T$ are 
\[
(x_2,x_4,x_6,x_8)=(0,11,2,0) 
\]
which is the same as ${\mathbf c}_{25}$ in~\cite[p.~139]{O92}.
Hence, $W_{14,8}$ is of Type ${\mathbf c}_{25}$
in the sense of~\cite{O92}.
It is claimed in~\cite[Theorem 3.6]{O92} that
a matrix of Type ${\mathbf c}_{25}$ is equivalent
to some matrix of  Type ${\mathbf c}_{i}$ $(i \ne 25)$.
This is an error.
\end{rem}


Among the weighing matrices of order $17$ and weight $9$,
Ohmori and Miyamoto~\cite{OM98} claimed to classify those with
intersection number $8$, and they found exactly $925$ such
matrices. However, we verified that only $517$ of the $2360$ weighing
matrices of order $17$ and weight $9$ have intersection number $8$.
Since their list of $925$ weighing matrices is not available,
we are unable to compare their result with ours.

\section{Other orders and weights}\label{Sec:other}

For orders $n \ge 13$ and weights $k \ge 6$, 
we classified
weighing matrices of some orders $n$ and weights $k$
listed in Table~\ref{Tab:O}
using the classification of maximal
self-orthogonal $\FF_p$-codes of length $n$ given in 
Table~\ref{Tab:MSO}.
Since approaches are similar to that used in 
Section~\ref{Sec:12},
we only list in Table~\ref{Tab:O} the numbers $\#$
of inequivalent weighing matrices for which we classified,
and the primes $p$.
Also, we list in the same table the orders and weights
for which we checked the known classifications
by our classification method, along with references.

\begin{table}[thbp]
\caption{Other orders and weights}
\label{Tab:O}
\begin{center}
{\small
\begin{tabular}{c|c|c|l||c|c|c|l}
\noalign{\hrule height0.8pt}
$(n,k)$ & $\#$ & $p$ & \multicolumn{1}{c||}{References} &
$(n,k)$ & $\#$ & $p$ & \multicolumn{1}{c}{References} \\
\hline 
$(13, 9)$ & 8   & 3 &\cite{O93} &$(16,15)$ &  1     & 3 &\cite{CRS86} \\
$(16, 6)$ & 30  & 3 &           &$(18, 9)$ &  11891 & 3 & \\
$(16, 9)$ & 704 & 3 &           &$(20, 6)$ &  49    & 3 & \\
$(16,10)$ & 670 & 5 &           &$(20,18)$ &  53    & 3 & \\
$(16,12)$ & 279 & 3 &           &$(24, 6)$ &  190   & 3 & \\
\noalign{\hrule height0.8pt}
   \end{tabular}
}
\end{center}
\end{table}




\end{document}